\documentclass{amsart}

\usepackage{amsfonts,amsmath,amssymb}

\input xy
\xyoption{all}

\newcommand{\M}{\mathcal M}

\newcommand{\G}{\mathcal G}
\newcommand{\w}{\omega}
\newcommand{\IR}{\mathbb R}

\newcommand{\I}{\mathcal I}
\newcommand{\J}{\mathcal J}
\newcommand{\A}{\mathcal A}

\newcommand{\F}{\mathcal F}
\newcommand{\K}{\mathcal K}
\newcommand{\C}{\mathcal C}
\newcommand{\U}{\mathcal U}
\newcommand{\V}{\mathcal V}
\newcommand{\W}{\mathcal W}
\newcommand{\HH}{\mathcal H}
\newcommand{\II}{\mathbb I}
\newcommand{\IZ}{\mathbb Z}
\newcommand{\IN}{\mathbb N}
\newcommand{\IQ}{\mathbb Q}

\newcommand{\add}{\mathrm{add}}
\newcommand{\cov}{\mathrm{cov}}
\newcommand{\non}{\mathrm{non}}
\newcommand{\cof}{\mathrm{cof}}
\newcommand{\id}{\mathrm{id}}

\newcommand{\e}{\varepsilon}
\newcommand{\Ra}{\Rightarrow}

\newcommand{\GG}{\mathcal G}

\newcommand{\Z}{\mathcal Z}
\newcommand{\Tau}{\mathcal T}
\newcommand{\diam}{\mathrm{diam}}
\newcommand{\NWD}{\mathrm{NWD}}

\newtheorem{theorem}{Theorem}[section]
\newtheorem{lemma}[theorem]{Lemma}

\newtheorem{corollary}[theorem]{Corollary}

\newtheorem{problem}[theorem]{Problem}
\newtheorem{example}[theorem]{Example}

\newtheorem{proposition}[theorem]{Proposition}
\newtheorem{claim}[theorem]{Claim}
\theoremstyle{definition}
\newtheorem{definition}[theorem]{Definition}

\title{Topologically invariant $\sigma$-ideals on the Hilbert cube}
\author{Taras Banakh, Micha\l\ Morayne, Robert Ra\l owski, Szymon \.Zeberski}
\address[Taras Banakh]{Department of Mathematics, Ivan Franko National University of Lviv, Ukraine, and\newline
Institute of Mathematics, Jan Kochanowski University, Kielce, Poland}
\email{t.o.banakh@gmail.com}
\address[Micha\l\ Morayne, Robert Ra\l owski, Szymon \. Zeberski]{Institute of Mathematics and Computer Science, Wroc\l aw University of Technology, Wroc\l aw, Poland}
\email{michal.morayne@pwr.wroc.pl, robert.ralowski@pwr.wroc.pl, szymon.zeberski@pwr.wroc.pl}
\subjclass{03E15; 03E17; 54H05; 55M10; 57N20}
\keywords{Topologically invariant ideal, cardinal characteristics, meager set, Cantor set, Hilbert cube}
\thanks{The work has been partially financed by NCN means granted by decision DEC-2011/01/B/ST1/01439. }

\begin{document}
\begin{abstract} We study and classify topologically invariant $\sigma$-ideals with Borel base on  the Hilbert cube and evaluate their cardinal characteristics. One of the results of this paper solves (positively) a known problem whether the minimal cardinalities of the families of Cantor sets
covering the unit interval and the Hilbert cube are the same.
\end{abstract}

\maketitle

\section{Introduction and survey of principal results}

In this paper we study properties of topologically invariant $\sigma$-ideals with Borel base on the Hilbert cube $\II^\w=[0,1]^\w$. In particular, we evaluate the cardinal characteristics of such $\sigma$-ideals. One of the results of this paper solves (positively) a known problem whether the minimal cardinalities of the families of Cantor sets
covering the unit interval and the Hilbert cube are the same.


Topologically invariant ideals in the finite dimensional case were considered in paper \cite{BMRZ} devoted to studying topologically invariant $\sigma$-ideals with Borel base on Euclidean spaces $\IR^n$. To present the principal results, we need to recall some definitions.

A non-empty family $\I$ of subsets of a set $X$ is called an {\em ideal on $X$} if $\I$ is hereditary with respect to taking subsets and $\I$ is additive (in the sense that $A\cup B\in\I$ for any subsets $A,B\in\I$).
An ideal $\I$ on $X$ is called a {\em $\sigma$-ideal} if for each countable subfamily $\A\subseteq\I$ the union $\bigcup\A$ belongs to $\I$. An ideal $\I$ on $X$ will be called {\em non-trivial} if $\I$ contains some uncountable subset of $X$ and $\I$ does not coincide with the ideal $\mathcal P(X)$ of all subsets of $X$. Each family $\F$ of subsets of a set $X$ generates the $\sigma$-ideal $\sigma\F$ consisting of subsets of countable unions of sets from the family $\F$.

A subset $A$ of a topological space $X$ has {\em the Baire property} (briefly, is a {\em BP-set}) if there is an open set $U\subseteq X$ such that the symmetric difference $A\triangle U=(A\setminus U)\cup(U\setminus A)$ is meager in $X$ (i.e., is a countable union of nowhere dense subsets of $X$).
A subfamily $\mathcal B\subseteq\I$ is called a {\em base} for $\I$ if each set $A\in\I$ is contained in some set $B\in\mathcal B$. We shall say that an ideal $\I$ on a Polish space $X$ has {\em $\sigma$-compact base} (resp. {\em Borel base}, {\em analytic base}, {\em BP-base}) if $\I$ has a base consisting of $\sigma$-compact (resp. Borel, analytic, BP-) subsets of $X$. Let us recall that a subset $A$ of a Polish space $X$ is {\em analytic} if $A$ is the image of a Polish space under a continuous map.  It is well-known that each Borel subset of a Polish space $X$ is analytic and each analytic subset of $X$ has the Baire property. Thus, for an ideal $\I$ on a Polish space $X$ we have the following implications:
$$\mbox{$\I$ has  $\sigma$-compact base $\Ra$ $\I$ has  Borel base $\Ra$ $\I$ has  analytic base $\Ra$ $\I$ has BP-base}.$$
Classical examples of $\sigma$-ideals with Borel base on the real line $\IR$ are the ideal $\M$ of meager subsets and the ideal $\mathcal N$ of Lebesgue null subsets of $\IR$. One of the differences between these ideals is that the ideal $\M$ is topologically invariant while $\mathcal N$ is not.

We shall say that an ideal $\I$ on a topological space $X$ is {\em topologically invariant} if $\I$ is preserved by homeomorphisms of $X$ in the sense that $\I=\{h(A):A\in\I\}$ for each homeomorphism $h:X\to X$ of $X$.

In \cite{BMRZ} we proved that the ideal $\M$ of meager subsets of a Euclidean space $\IR^n$ is the largest topologically invariant $\sigma$-ideal with BP-base on $\IR^n$. This is not true anymore for the Hilbert cube $\II^\w$ as shown by the $\sigma$-ideal $\sigma\mathcal D_{0}$ of countable-dimensional subsets of $\II^\w$. The $\sigma$-ideal $\sigma\mathcal D_0$ is generated by all zero-dimensional subspaces of $\II^\w$ and has a base consisting of countable-dimensional $G_{\delta\sigma}$-sets. It is clear that $\sigma\mathcal D_0\not\subseteq\M$. So, $\M$ is not the largest non-trivial $\sigma$-ideal with Borel base on $\II^\w$. Nonetheless, the ideal $\M$ has the following maximality property.

\begin{theorem}\label{t1.1} The ideal $\M$ of meager subsets of the Hilbert cube $\II^\w$ is:
\begin{enumerate}
\item a maximal non-trivial topologically invariant ideal with BP-base on $\II^\w$, and
\item the largest non-trivial topologically invariant ideal with $\sigma$-compact base on $\II^\w$.
\end{enumerate}
\end{theorem}

\begin{proof} (1) Given a non-trivial topologically invariant ideal $\I\supseteq \M$ with BP-base on $\II^\w$ we should prove that $\I=\M$. Assume that $\I$ contains some subset $A\in\I\setminus\M$. Since $\I$ has  BP-base, we can additionally assume that the set $A$ has the Baire property in $\II^\w$.
Being non-meager, the BP-set $A$ contains a $G_\delta$-subset $G_U\subseteq A$, dense in some non-empty open set $U\subseteq Q$. The compactness and the topological homogeneity of the Hilbert cube (see e.g. \cite[6.1.6]{Mill}) allow us to find a finite sequence of homeomorphisms $h_1,\dots,h_n:\II^\w\to \II^\w$ such that $\II^\w=\bigcup_{i=1}^nh_i(U)$. The topological invariance of the ideal $\I$ guarantees that the dense $G_\delta$-set $G=\bigcup_{i=1}^nh_i(G_U)$ belongs to the ideal $\I$. Since $\II^\w\setminus G\in\M\subseteq\I$ is meager, we conclude that $\II^\w=G\cup(\II^\w\setminus G)\in\I$, which means that the ideal $\I$ is trivial.
\smallskip

(2) Next, assume that $\I$ is a non-trivial topologically invariant ideal on $\II^\w$ with $\sigma$-compact base. To show that $\I\subseteq\M$, it suffices to check that each $\sigma$-compact set $K\in\I$ is meager in $\II^\w$. Assuming that $K$ is not meager and applying Baire Theorem, we conclude that $K$ contains a non-empty open subset $U\subseteq\II^\w$. By the compactness and the topological homogeneity of $\II^\w$ there are homeomorphisms $h_1,\dots,h_n$ of $\II^\w$ such that $\II^\w=\bigcup_{i=1}^n h_i(U)$. Now the topological invariance and the additivity of $\I$ imply that $\II^\w\in\I$, which means that the ideal $\I$ is trivial.
\end{proof}

In \cite{BMRZ} we proved that the family of all non-trivial topologically invariant $\sigma$-ideals with analytic base on a Euclidean space $\IR^n$ contains the smallest element, namely the $\sigma$-ideal $\sigma\C_0$ generated by so called tame Cantor sets in $\IR^n$. A similar fact holds also for topologically invariant ideals with an analytic base on the Hilbert cube $\II^\w$.

By a {\em Cantor set} in $\II^\w$ we understand any subset $C\subseteq \II^\w$ homeomorphic to the Cantor cube $\{0,1\}^\w$. By Brouwer's characterization \cite{Ke} of the Cantor cube, a closed subset $C\subseteq\II^\w$ is a Cantor set if and only if $C$ is zero-dimensional and has no isolated points.
A Cantor set $A\subseteq \II^\w$ is called {\em minimal} if for each Cantor set $B\subseteq\II^\w$ there is a homeomorphism $h:\II^\w\to\II^\w$ such that $h(A)\subseteq B$.

Minimal Cantor sets in the Hilbert cube $\II^\w$ can be characterized as Cantor $Z_\w$-sets. Let us recall that a closed subset $A$ of a topological space $X$ is called a {\em $Z_n$-set} in $X$ for $n\le\w$ if the set $\{f\in C(\II^n,X):f(\II^n)\cap A=\emptyset\}$ is dense in the space $C(\II^n,X)$ of all continuous functions from $\II^n$ to $X$, endowed with the compact-open topology. A closed subset $A$ of a topological space $X$ is called a {\em $Z$-set} in $X$ if for any open cover $\U$ of $X$ there is a continuous map $f:X\to X\setminus A$, which is {\em $\U$-near} to the identity map in the sense that for each $x\in X$ the set $\{f(x),x\}$ is contained in some set $U\in\U$. It is clear that a subset of the Hilbert cube is a $Z$-set in $\II^\w$ if and only if it is a $Z_\w$-set in $\II^\w$.
By the $Z$-Set Unknotting Theorem 11.1 in \cite{Chap}, any two Cantor $Z$-sets $A,B\subseteq\II^\w$ are {\em abmiently homeomorphic}, which means that there is a homeomorphism $h:\II^\w\to\II^\w$ such that $h(A)=B$. This implies that a Cantor set $C\subseteq\II^\w$ is minimal if and only if it is a $Z$-set in $\II^\w$, see Proposition~\ref{Zmin}.

By $\C_0$ we denote the family of all minimal Cantor sets in $\II^\w$ and by
$\sigma\C_0$ the $\sigma$-ideal generated by the family $\C_0$. Observe that $\sigma\C_0$ coincides with the $\sigma$-ideal generated by zero-dimensional $Z$-sets in $\II^\w$. The following theorem shows that $\sigma\C_0$ is the smallest non-trivial $\sigma$-ideal with analytic base on $\II^\w$.

\begin{theorem}\label{t1.2} The family $\C_0$ (the $\sigma$-ideal $\sigma\C_0$) is contained in each topologically invariant non-trivial ($\sigma$-)ideal $\I$ with analytic base on $\II^\w$.
\end{theorem}

\begin{proof}

Let $C\in\C_0$. By the hypothesis there exists an uncountable analytic set $A\in\I$. It is known that $A$ contains a copy $C_1$
of a Cantor set. By the minimality of $C$ there exists a homeomorphism $h$ of the Hilbert cube such that $h(C)\subseteq C_1\subseteq A$.
Hence $h(C)\in\I$. By the topological invariance of $\I$ also $C\in\I$.
\end{proof}

As we know, on the Hilbert cube there are non-trivial topologically invariant $\sigma$-ideals with Borel base, which are not contained in the ideal $\M$ of meager sets. It turns out that among such $\sigma$-ideals there is the smallest one. It is denoted by $\sigma\GG_0$ and is generated by minimal dense $G_\delta$-subsets of $\II^\w$.

A dense $G_\delta$-subset $A$ of a Polish space $X$ will be called {\em minimal} if for each dense $G_\delta$-set $B\subseteq X$ there is a homeomorphism $h:X\to X$ such that $h(A)\subseteq B$.
By \cite{BR2}, any two minimal dense $G_\delta$-subsets of $\II^\w$ are ambiently homeomorphic. Minimal dense $G_\delta$-sets in $\II^\w$ were characterized in \cite{BR2} as dense tame $G_\delta$-set.
To introduce tame $G_\delta$-sets in the Hilbert cube we need some additional notions.

A family $\V$ of subsets of a topological space $X$ is called {\em vanishing} if for any open cover $\U$ of $X$ the subfamily $\{V\in\V:\forall U\in\U\;\;V\not\subseteq U\}$ is locally finite in $X$.

An open subset $U$ of $\II^\w$ is called a {\em tame open ball} if
\begin{itemize}
\item its closure $\bar U$ in $\II^\w$ is homeomorphic to the Hilbert cube;
 \item its boundary $\partial U$ in $\II^\w$ is homeomorphic to the Hilbert cube;
\item $\partial U$ is a $Z$-set in $\bar U$ and in $\II^\w\setminus U$.
\end{itemize}
 By \cite[12.2]{Chap}, tame open balls form a base of the topology of the Hilbert cube.

A subset $U$ of $\II^\w$ is called a {\em tame open set} in $\II^\w$ if $U=\bigcup\U$ for some  vanishing family $\U$ of tame open balls with pairwise disjoint closures in $\II^\w$. The family $\U$ is unique and coincides with the family $\C(U)$ of all connected components of $U$. By $\bar\C(U)=\{\bar C:C\in\C(U)\}$ we shall denote the (disjoint) family of closures of the connected components of the set $U$.

A subset $G$ of $\II^\w$ is called a {\em tame $G_\delta$-set} in $\II^\w$ if $G=\bigcap_{n\in\w}U_n$ for some sequence $(U_n)_{n\in\w}$ of tame open sets in $\II^\w$ such that $\bigcup\bar\C(U_{n+1})\subseteq U_n$ for every $n\in\w$ and the family $\bigcup_{n\in\w}\bar \C(U_n)$ is vanishing in $\II^\w$.

By Theorem 4 of \cite{BR2}, a dense $G_\delta$-set in $\II^\w$ is minimal if and only if it is dense tame $G_\delta$ in $\II^\w$.

Denote by $\GG_0$ the family of all minimal dense $G_\delta$-sets in $\II^\w$ and by $\sigma\GG_0$ the $\sigma$-ideal generated by the family $\GG_0$. It is clear that $\GG_0\not\subseteq\M$. It turns out that the $\sigma$-ideal $\sigma\GG_0$ is the smallest topologically invariant $\sigma$-ideal with BP-base, which is not contained in the ideal $\M$ of meager subsets in $\II^\w$.

\begin{theorem}\label{t1.3} The family $\GG_0$ (the $\sigma$-ideal $\sigma\GG_0$) is contained in each topologically invariant ($\sigma$-)ideal $\I\not\subseteq\M$ with BP-base on $\II^\w$.
\end{theorem}

\begin{proof} If $\I\not\subseteq\M$, then we can find a non-meager set $A\in\I$. Repeating the argument from the proof of Theorem~\ref{t1.1}, we can show that the ideal $\I$ contains a dense $G_\delta$-subset $G$ of $\II^\w$. To check that $\GG_0\subseteq \I$, fix any minimal dense $G_\delta$-set $M\subseteq \II^\w$ and find a homeomorphism $h$ of $\II^\w$ such that $h(M)\subseteq G\in\I$. Then $h(M)\in\I$ and $M\in\I$ by the topological invariance of $\I$.
\end{proof}

In light of Theorem~\ref{t1.3}, it is important to study the properties of the $\sigma$-ideal $\sigma\GG_0$ and how it is related to other $\sigma$-ideals. Since each minimal dense $G_\delta$-set in $\II^\w$ is zero-dimensional, the ideal $\sigma\GG_0$ is contained in the $\sigma$-ideal $\sigma\mathcal D_0$ generated by the family $\mathcal D_0$ of all zero-dimensional subspaces of $\II^\w$. By \cite[1.5.8]{En2}, the $\sigma$-ideal $\sigma\mathcal D_0$ contains the $\sigma$-ideal $\sigma\overline{\mathcal D}_{<\w}$ generated by the family $\overline{\mathcal D}_{<\w}$ of all closed finite-dimensional subsets of $\II^\w$.

\begin{theorem}\label{t1.4} The family $\overline{\mathcal D}_{<\w}$ is contained in each topologically invariant ideal $\I\not\subseteq\M$ with BP-base on $\II^\w$. Consequently, $\sigma\overline{\mathcal D}_{<\w}\subsetneq\sigma\GG_0\subseteq \sigma\mathcal D_{0}$.
\end{theorem}

Theorem~\ref{t1.4} will be proved in Section~\ref{s:t1.4}.
Theorems~\ref{t1.2} and \ref{t1.3} will help us to evaluate the cardinal characteristics of an arbitrary topologically invariant $\sigma$-ideal with analytic base on the Hilbert cube.

Given an ideal $\I$ on a set $X=\bigcup\I\notin \I$, we shall consider the following four cardinal characteristics of $\I$:
$$\begin{aligned}
&\add(\I)=\min\{|\A|:\A\subseteq\I,\;\;\textstyle{\bigcup}\A\notin\I\},\\
&\non(\I)=\min\{|A|:A\subseteq X,\;\;A\notin\I\},\\
&\cov(\I)=\min\{|\A|:\A\subseteq\I,\;\textstyle{\bigcup}\A=X\},\\
&\cof(\I)=\min\{|\A|:\A\subseteq\I,\;\;\forall B\in\I\;\;\exists A\in\A\;\;(B\subseteq A)\}.
\end{aligned}
$$

In fact, these four cardinal characteristics can be expressed using the following two cardinal characteristics defined for any pair $\I\subseteq\J$ of ideals:
$$\begin{aligned}
&\add(\I,\J)=\min\{|\A|:\A\subseteq\I,\;\;\textstyle{\bigcup}\A\notin\J\} \mbox{ \ and}\\
&\cof(\I,\J)=\min\{|\A|:\A\subseteq\J,\;\;\forall B\in\I\;\exists A\in\A\;\;(B\subseteq A)\}.
\end{aligned}
$$
Namely,
$$\add(\I)=\add(\I,\I),\;\;\non(\I)=\add(\F,\I),\;\;
\cov(\I)=\cof(\F,\I),\;\;\cof(\I)=\cof(\I,\I)$$where $\F$ stands for the ideal of finite subsets of $X$.

The cardinal characteristics of the $\sigma$-ideal $\M$ have been thoroughly studied in Set Theory, see \cite{BJ} or \cite{Blass}.
They fit into the following (piece of Cicho\'n's) diagram in which an arrow $a\to b$ indicates that $a\le b$ in ZFC:

$$\xymatrix{
&&\non(\M)\ar[r]&\cof(\M)\ar@{=}[r]&\max\{\non(\M),\mathfrak d\}\ar[r]&\mathfrak c\\
&&\mathfrak b\ar[u]\ar[r]&\mathfrak d\ar[u]\\
\w_1\ar[r]&\min\{\mathfrak b,\cov(\M)\}\ar@{=}[r]&\add(\M)\ar[r]\ar[u]&\cov(\M)\ar[u]
}
$$

Here $$
\begin{aligned}
\mathfrak b&=\min\{|B|:B\subseteq\w^\w,\;\forall x\in\w^\w\;\exists y\in B\;\;(y\not\le^*x)\}\mbox{ and}\\
\mathfrak d&=\min\{|D|:D\subseteq\w^\w,\;\forall x\in\w^\w\;\exists y\in D\;\;(x\le^*y)\}
\end{aligned}
$$are the {\em bounding} and {\em dominating} numbers \cite{vD}, \cite{Vau}, \cite{Blass} (the notation $x\le^* y$ means $x(n)\le y(n)$ for all but finitely many numbers $n$). The precise position of the small uncountable cardinals $\mathfrak b$ and $\mathfrak d$ in the interval $[\w_1,\mathfrak c]$ depends on additional axioms of ZFC, see \cite{BJ}. The same concerns the cardinal characteristics of the ideal $\M$ on the real line: their values depend on axioms, too, see  \cite{BartCich}.

The cardinal characteristics of the ideal $\M$ will be used to evaluate the cardinal characteristics of the $\sigma$-ideals $\sigma\C_0$ and $\sigma\GG_0$ in the following theorems which are principal results of this paper. These two theorems will be proved in Sections~\ref{s:t1.5} and \ref{s:t1.6}, respectively.

\begin{theorem}\label{t1.5} The $\sigma$-ideal $\sigma\C_0$ on the Hilbert cube has cardinal characteristics:
\begin{enumerate}
\item $\cov(\sigma\C_0)=\cov(\M)$,
\item $\non(\sigma\C_0)=\non(\M)$,
\item $\add(\sigma\C_0,\M)=\add(\M)$,
\item $\cof(\sigma\C_0,\M)=\cof(\M)$.
\end{enumerate}
\end{theorem}

\begin{theorem}\label{t1.6} The $\sigma$-ideal $\sigma\GG_0$ on the Hilbert cube has cardinal characteristics:
$$\w_1\le\add(\sigma\GG_0)\le \cov(\sigma\GG_0)\le\add(\M)\le\cof(\M)\le\non(\sigma\GG_0)\le\cof(\sigma\GG_0)\le\mathfrak c.$$
\end{theorem}

Theorems~\ref{t1.2}, \ref{t1.3}, \ref{t1.5} and \ref{t1.6} imply the following corollary.

\begin{corollary}\label{c1.7} Let $\I$ be a non-trivial topologically invariant $\sigma$-ideal $\I$ with analytic base on the Hilbert cube.
\begin{enumerate}
\item If $\I\subseteq\M$, then $\cov(\I)=\cov(\M)$, $\non(\I)=\non(\M)$, $\add(\I)\le\add(\M)$, and $\cof(\I)\ge\cof(\M)$.
\item If $\I\not\subseteq\M$, then $\add(\I)\le\cov(\I)\le\cov(\sigma\GG_0)\le\add(\M)\le\cof(\M)\le\non(\sigma\GG_0)\le \non(\I)\le\cof(\I).$
\end{enumerate}
\end{corollary}

Corollary~\ref{c1.7} implies that the cardinal characteristics of any non-trivial topologically invariant $\sigma$-ideals $\I\subseteq\M$ and $\J\not\subseteq\M$ with analytic base on the Hilbert cube fit in the following variant of Cicho\'n's diagram:
$$
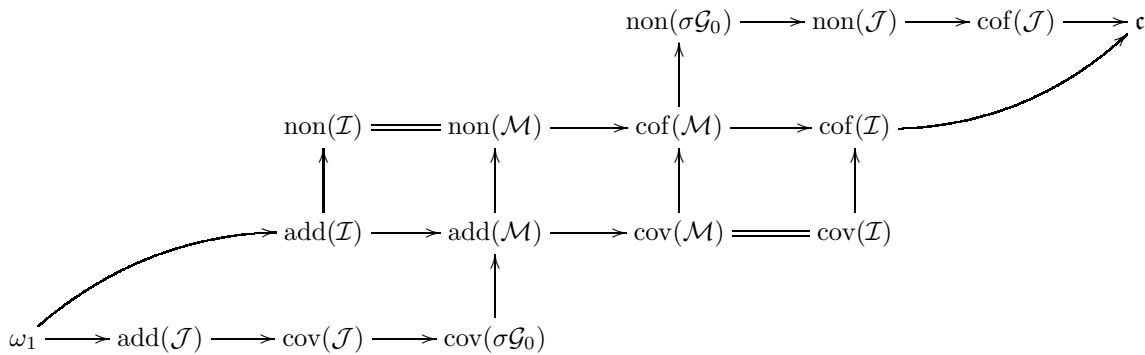
\xymatrix{
&&&&\non(\sigma\GG_0)\ar[r]&\non(\J)\ar[r]&\cof(\J)\ar[r]&\mathfrak c\\
&&\non(\I)\ar@{=}[r]&\non(\M)\ar[r]&\cof(\M)\ar[r]\ar[u]&\cof(\I)\ar@/_1.1pc/[urr]\\
&&\add(\I)\ar[u]\ar[r]&\add(\M)\ar[u]\ar[r]&\cov(\M)\ar[u]\ar@{=}[r]&\cov(\I)\ar[u]&\\
\w_1\ar@/^1.1pc/[urr]\ar[r]&\add(\J)\ar[r]&\cov(\J)\ar[r]&\cov(\sigma\GG_0)\ar[u]\\
}
$$
\smallskip

The following example shows that the inequalities $\add(\I),\cov(\J)\le\add(\M)$ and $\cof(\M)\le\cof(\I),\non(\J)$ in this diagram can be strict. By an {\em arc} in $\II^\w$ we understand any subset $A\subseteq\II^\w$ homeomorphic to the closed interval $\II=[0,1]$.

\begin{example} \begin{enumerate}
\item The $\sigma$-ideal $\I$ generated by arcs in $\II^\w$ has\newline $\add(\I)=\w_1$, $\cov(\I)=\cov(\M)$, $\non(\I)=\non(\M)$ and $\cof(\I)=\mathfrak c$.
\item The $\sigma$-ideal $\J=\sigma\mathcal D_0\not\subseteq \M$ of countable dimensional subsets of\/ $\II^\w$ has\newline $\add(\J)=\cov(\J)=\w_1$ and $\non(\J)=\cof(\J)=\mathfrak c$.
\end{enumerate}
\end{example}

\begin{proof} 1. The equalities $\cov(\I)=\cov(\M)$ and $\non(\I)=\non(\M)$ follow from Corollary \ref{c1.7}.
To see that $\add(\I)=\w_1$ and $\cof(\I)=\mathfrak c$, observe that the Hilbert cube contains continuum many pairwise disjoint arcs and each arc contains at most countably many pairwise disjoint subarcs.
\smallskip

To see that $\non(\sigma\mathcal D_0)=\mathfrak c$, observe that each subset $A\subseteq\II^\w$ of cardinality $|A|<\mathfrak c$ is zero-dimensional. The equality  $\cov(\sigma\mathcal D_0)=\w_1$ is an old result of Smirnov, see \cite[5.1.B]{En2}.
\end{proof}


Next, we describe certain topologically invariant $\sigma$-ideals $\I$ with $\sigma$-compact base on the Hilbert cube whose cardinal characteristics coincide with the respective cardinal characteristics of the ideal $\M$. In the following definition for a compact topological space $X$ by $\K(X)$ we shall denote the family of all compact subsets of $X$ endowed with the Vietoris topology. The hyperspace $\K(X)$ is partially ordered by the inclusion relation.

\begin{definition}\label{d1.9} An ideal $\I$ on a topological space $X$
\begin{itemize}
\item is a {\em $G_\delta$-ideal} if the set $\I\cap\K(X)$ is of type $G_\delta$ in the hyperspace $\K(X)$;
\item has the {\em Solecki property $(*)$} if for any countable family $\A\subset\I\cap\K(X)$ the union $\bigcup\A$ is contained in a $G_\delta$-subset $G\subseteq X$ such that $\K(G)\subset \I$;
\item is a {\em $\sigma^{(*)}$-ideal\/} if there is a sequence $(\I_n)_{n\in\w}$ of $G_\delta$-ideals with the Solecki property $(*)$ on $X$ such that $\I=\{A\in\mathcal P(X):A\subset \bigcup\A$ for a countable subfamily $\A\subset\bigcup_{n\in\w}\I_n\cap\K(X)\}$.
\end{itemize}
\end{definition}

Definition~\ref{d1.9} implies that each $\sigma^{(*)}$-ideal is a $\sigma$-ideal with $\sigma$-compact base. It is easy to see that the countable intersection $\bigcap_{n\in\w}\I_n$ of $G_\delta$-ideals with the Solecki property $(*)$ is a $G_\delta$-ideal with the Solecki property $(*)$. This implies that a  finite intersection of $\sigma^{(*)}$-ideals is a $\sigma^{(*)}$-ideal.

By Proposition 2.1 of \cite{Sol}, an ideal $\I$ with the Solecki property $(*)$ on a compact metrizable space $X$ is a $G_\delta$-ideal if and only if the intersection $\I\cap \K(X)$ is analytic or coanalytic in the hyperspace $\K(X)$. This implies that many natural ideals with the Solecki property $(*)$ are $G_\delta$-ideals. A list of such ideals can be found in \cite{Sol}.
In \cite{Sol} S.Solecki established an important fact on the Tukey reducibility of $G_\delta$-ideals with the property $(*)$.

Let $A,B$ be two partially ordered sets. We say that $A$ is {\em Tukey reducible} to $B$ if there is a function $f:A\to B$ such that for each $b\in B$ the set $\{a\in A:f(a)\le b\}$ is upper bounded in $A$. The function $f$ will be called a {\em Tukey reduction} of $A$ to $B$. In \cite{Sol} Solecki proved that for each $G_\delta$-ideal $\I$ with the Solecki property $(*)$ on a compact metrizable space $X$ the partially ordered set $\I\cap\K(X)$ is Tukey reducible to the ideal $\NWD$ of nowhere dense sets in the Cantor cube $2^\w$. The same fact is true for $\sigma^{(*)}$-ideals.

\begin{proposition}\label{solecki} Every $\sigma^{(*)}$-ideal $\I$ on a compact metrizable space $X$ is Tukey reducible to the ideal $\NWD$, and has $\add(\I)\ge\add(\M)$ and $\cof(\I)\le\cof(\M)$.
\end{proposition}

\begin{proof} Since $\I$ is a $\sigma^{(*)}$-ideal, there is a sequence of $G_\delta$-ideals $(\I_n)_{n\in\w}$ with the Solecki property $(*)$ on $X$ such that  $\I=\{A\in\mathcal P(X):A\subset\bigcup\A$ for some countable family $\A\subset\bigcup_{n\in\w}\I_n\cap\K(X)\}$. Using the Axiom of Choice we can choose a function $f:\I\to \prod_{n\in\w}(\I_n\cap \K(X))^\w$ assigning to each set $A\in\I$ a double sequence $(A_{n,m})_{n,m\in\w}\in\prod_{n\in\w}(\I_n\cap \K(X))^\w$ such that $A\subset\bigcup_{n,m\in\w}A_{n,m}$. The function $f$ witnesses that the ideal $\I$ is Tukey reducible to the partially ordered set $\prod_{n\in\w}(\I_n\cap \K(X))^\w$.
Here the countable product $\prod_{n\in\w}P_n$ of partially ordered sets is endowed with coordinatewise partial order $(x_n)_{n\in\w}\le(y_n)_{n\in\w}$ iff $x_n\le y_n$ for all $n\in\w$.

By Theorem 4.1 of \cite{Sol}, for every $n\in\w$ there is a Tukey reduction $g_n:\I_n\cap\K(X)\to \NWD$. The maps $(g_n)_{n\in\w}$ produce a Tukey reduction $g=(g_n^\w)_{n\in\w}:\prod_{n\in\w}(\I_n\cap\K(X))^\w\to\NWD^{\w\times\w}$. By Theorem 3B(a) of \cite{fremlin}, the partially ordered sets $\NWD^\w$ and $\NWD$ are isomorphic. Consequently, there is a Tukey reduction $h:\NWD^{\w\times\w}\to\NWD$. The composition $\varphi=h\circ g\circ f:\I\to\NWD$ is a Tukey reduction of the ideal $\I$ to the ideal $\NWD$.
\smallskip

By a standard method it can be shown that the Tukey reducibility of $\I$ to $\NWD$ implies the inequalities $\add(\I)\ge\add(\M)$ and $\cof(\I)\le\cof(\M)$, see Theorem 1J of \cite{fremlin}.
\end{proof}

Proposition~\ref{solecki} and Corollary~\ref{c1.7}(1) imply:

\begin{corollary}\label{t1.10} Any non-trivial topologically invariant $\sigma^{(*)}$-ideal $\I$ on $\II^\w$ has  $$\add(\I)=\add(\M),\;\;\cov(\I)=\cov(\M),\;\;\non(\I)=\non(\M) \mbox{ \ and \ }\cof(\I)=\cof(\M).$$
\end{corollary}

Now we consider some examples of topologically invariant $\sigma^{(*)}$-ideals on $\II^\w$. For any cardinals $n,m\le\w$ let
\begin{itemize}
\item $\mathcal Z_n$ be the family of $Z_n$-sets in $\II^\w$;
\item $\overline{\mathcal D}_{m}$ be the family of closed subsets $A\subset \II^\w$ of topological dimension $\dim(A)\le m$;
\item $\sigma(\mathcal Z_n\cap\overline{\mathcal D}_{m})$ be the $\sigma$-ideal generated by the family $\mathcal Z_n\cap\overline{\mathcal D}_{m}$.
\end{itemize}
Observe that $\sigma(\mathcal Z_\w\cap\overline{\mathcal D}_0)$ coincides with the $\sigma$-ideal $\sigma\C_0$ while $\sigma(\mathcal Z_0\cap\overline{\mathcal D}_\w)$ coincides with the ideal $\M$ of meager subsets of $\II^\w$.

In the following corollary we calculate the cardinal characteristics of the $\sigma$-ideals $\sigma(\Z_n\cap\overline{\mathcal D}_m)$, $n,m\le\w$, thus answering Problem 2.6 of \cite{BCZ}.
The same answer was obtained independently by Solecki \cite{Solprivat}.

\begin{corollary}\label{c1.11} For every $n,m\le\w$ the $\sigma$-ideal $\I=\sigma(\Z_n\cap\overline{\mathcal D}_m)$ is a $\sigma^{(*)}$-ideal on $\II^\w$ with cardinal characteristics
$$\add(\I)=\add(\M),\;\;\cov(\I)=\cov(\M),\;\;\non(\I)=\non(\M),\;\;
\cof(\I)=\cof(\M).$$
\end{corollary}

\begin{proof} First we show that the ideal $\I$ is a $G_\delta$-ideal with the Solecki property $(*)$. Given a countable subfamily $\A\subset\I\cap \K(\II^\w)$, we shall find a $G_\delta$-set $G\subset\II^\w$ such that $\bigcup\A\subset G$ and $\K(G)\subset \Z_n$. We lose no generality assuming that $\A\subset \Z_n\cap\overline{\mathcal D}_m$. By the Sum Theorem \cite[1.5.2]{En2}, the union $\bigcup\A$ has topological dimension $\le m$ and by the Enlargement Theorem \cite[1.5.10]{En2}, the set $\bigcup\A$ is contained in a $G_\delta$-subset $G'\subset\II^\w$ of topological dimension $\le m$.

It follows that for every $Z_n$-set $A\in\A$ the set $F_A=\{f\in C(\II^n,\II^\w):f(\II^n)\cap A=\emptyset\}$ is open and dense in the function space $C(\II^n,\II^\w)$. The Baire Theorem guarantees that the set $F_\A=\bigcap_{A\in\A}F_A$ is a dense $G_\delta$-set in $C(\II^n,\II^\w)$. Fix a countable dense subset $\F\subset F_\A$ and observe that $G=G'\setminus\bigcup_{f\in F_\A}f(\II^n)$ is a required $G_\delta$-set in $\II^\w$ with $\K(G)\subset\Z_n\cap\overline{\mathcal D}_m$, which contains the union $\bigcup\A$. This witnesses that the $\sigma$-ideal $\I=\sigma(\Z_n\cap\overline{\mathcal D}_m)$ has the Solecki property $(*)$.

Next we show that $\I$ is a $G_\delta$-ideal. Observe that $\I\cap\K(\II^\w)=\mathcal Z_n\cap\overline{\mathcal D}_m$. By \cite{DMM}, the family $\overline{\mathcal D}_m$ is a $G_\delta$-set in the hyperspace $\K(\II^\w)$. To show that $\mathcal Z_n$ is a $G_\delta$-set in $\K(\II^\w)$, fix a countable base $(\U_n)_{n\in\w}$ of the topology of the function space $C(\II^n,\II^\w)$ and in each set $\U_n$ fix a function $f_n\in\U_n$. Then $\{f_n\}_{n\in\w}$ is a countable dense set in the function space $C(\II^n,\II^\w)$. It follows from
$$
\begin{aligned}
\Z_n&=\{A\in\K(\II^\w):\forall n\in\w\;\exists m\in\w\mbox{ such that
$f_m\in\U_n$ and $A\cap f_m(\II^n)=\emptyset$}\}=\\
&=\bigcap_{n\in\w}\bigcup_{f_m\in U_n}\{A\in\K(\II^\w):A\subset \II^\w\setminus f_m(\II^n)\},
\end{aligned}
$$
that the families $\Z_n$ is of type $G_\delta$ in the hyperspace $\K(\II^\w)$. Consequently, $\Z_n\cap\overline{\mathcal D}_m=\K(\II^\w)\cap\sigma(\Z_n\cap\overline{\mathcal D}_m)$ is a $G_\delta$-set in $\K(\II^\w)$ and $\I=\sigma(\Z_n\cap\overline{\mathcal D}_m)$ is a $G_\delta$-ideal on $\II^\w$.

Since $\I=\{A\in\mathcal P(\II^\w):A\subset\bigcup\A$ for a countable subfamily $\A\subset\I\cap\K(\II^\w)\}$, the ideal $\I$ is a $\sigma^{(*)}$-ideal. By Corollary~\ref{c1.11}, $\add(\I)=\add(\M)$, $\cov(\I)=\cov(\M)$, $\non(\I)=\non(\M)$, and $\cof(\I)=\cof(\M)$.
\end{proof}

\section{Some Properties of $Z$-sets in the Hilbert cube}

In this section we collect the necessary information on $Z$-sets in the Hilbert cube $\II^\w$.

We recall that a subset $A$ of a topological space $X$ is called a {\em $Z$-set} in $X$ if $A$ is closed in $X$ and for each open cover $\U$ of $X$ there is a map $f:X\to X\setminus A$, which is $\U$-near to the identity map $\id:X\to X$. All {\em maps} considered in this paper {\em are assumed to be continuous}. In contrast, {\em functions need not be continuous}. It follows that a subset of the Hilbert cube $\II^\w$ is a $Z$-set in $\II^\w$ if and only if it is a $Z_\w$-set in $\II^\w$.

A subset $A$ of a topological space $X$ is called a {\em $\sigma Z$-set} in $X$ is $A$ can be written as the countable union $A=\bigcup_{n=1}^\infty A_n$ of $Z$-sets.

The following known lemma (see \cite[Lemma 6.2.3(2)]{Mill}) gives many examples of $Z$-sets in the Hilbert cube.

\begin{lemma}\label{product-Z-set} For any closed proper subsets $A_n\subsetneqq \II$, $n\in\w$, the product $\prod_{n\in\omega}A_n$ is a $Z$-set in $\II^\w$.
\end{lemma}

We shall often use the following powerful homogeneity property of the Hilbert cube \cite[11.1]{Chap}.

\begin{theorem}[Z-Set Unknotting Theorem]\label{unknot} Any homeomorphism $h:A\to B$ between $Z$-sets $A,B$ in the Hilbert cube $\II^\w$ extends to a homeomorphism of\/ $\II^\w$.
\end{theorem}

A map $f:X\to Y$ will be called a {\em $Z$-embedding} if $f(X)$ is a $Z$-set in $Y$ and $f:X\to f(X)$ is a homeomorphism. The following universality property of the Hilbert cube was proved in \cite[11.2]{Chap}.

\begin{theorem}[Approximation by $Z$-embedding]\label{Z-approx} For any compact metrizable space $K$ the set of $Z$-embeddings is dense in the function space $C(K,\II^\w)$.
\end{theorem}

We shall apply the $Z$-Set Unknotting Theorem to prove the following tameness lemma.

\begin{lemma}\label{l2.4} For each zero-dimensional $Z$-set $A$ in the Hilbert cube $\II^\w$ and every open cover $\U$ of $\II^\w$ there is a finite cover $B_1,\dots,B_n$ of $A$ by open subsets of\/ $\II^\w$ such that
\begin{enumerate}
\item for every $i\le n$ the closure $\bar B_i$ is homeomorphic to $\II^\w$, is contained in some set $U\in\U$ and the boundary $\partial B_i$ of $B_i$ in $\II^\w$ is a $Z$-set in $\bar B_i$;
\item for any distinct numbers $i,j\le n$ the closures $\bar B_i$ and $\bar B_j$ are disjoint.\end{enumerate}
\end{lemma}

\begin{proof} This lemma is obvious for zero-dimensional sets contained in straight intervals of the form $[0,1]\times\{x_0\}\subseteq \II\times\II^{\IN}=\II^\w$. The general case can be reduced to this special case with help of the $Z$-set Unknotting Theorem~\ref{unknot}.
\end{proof}

\begin{proposition}\label{Zmin} A Cantor set $C$ in the Hilbert cube $\II^\w$ is a minimal Cantor set if and only if it is a $Z$-set in $\II^\w$.
\end{proposition}

\begin{proof} To prove the ``if'' part, assume that a Cantor set $C\subset\II^\w$ is a $Z$-set in $\II^\w$. To prove that $C$ is a minimal Cantor set, fix any Cantor set $B\subset\II^\w$.
We claim that $B$ contains a Cantor $Z$-set $K\subset B$.

Using Theorem~\ref{Z-approx}, in the function space $C(\II^\w,\II^\w)$ we can choose a countable dense subset $\{f_n\}_{n\in\w}$ consisting of $Z$-embeddings. If for some $n\in\w$ the intersection $B\cap f_n(\II^\w)$ is uncountable, then it contains a Cantor set $K\subset B\cap f_n(\II^\w)$ according to Souslin's Theorem \cite[29.1]{Ke}. Since $f_n(\II^\w)$ is a $Z$-set in $\II^\w$, the Cantor set $K\subset f_n(\II^\w)$ is a $Z$-set in $\II^\w$ too. If each intersection $B\cap f_n(\II^\w)$ is at most countable, then the set $B\setminus \bigcup_{n\in\w}f_n(\II^\w)$ is uncountable and hence contains a Cantor set $K\subset B\setminus \bigcup_{n\in\w}f_n(\II^\w)$ according to Souslin's Theorem  \cite[29.1]{Ke}. The density of the set $\{f_n\}_{n\in\w}$ in $C(\II^\w,\II^\w)$ guarantees that $K$ is a $Z$-set in $\II^\w$.
Using $Z$-set Unknotting Theorem~\ref{unknot}, we can find a homeomorphism $h$ of $\II^\w$ such that $h(C)=K\subset B$. This homeomorphism witnesses that $C$ is a minimal Cantor set in $\II^\w$.
\smallskip

The ``only if'' part follows from the existence of a Cantor $Z$-set in $\II^\w$.
\end{proof}

\section{Proof of Theorem~\ref{t1.4}}\label{s:t1.4}

We shall derive Theorem~\ref{t1.4} from five lemmas proved below. In these lemmas by $X$ we denote the Hilbert cube $\II^\w$ and by $\HH(X)$ its homeomorphism group endowed with the compact-open topology.
A neighborhood base of this topology at each $h\in\HH(X)$ consists of the sets
$$B(h,\U)=\{f\in\HH(X):(h,f)\prec\U\},$$
where $\U$ runs over all open covers of $X$. For a cover $\U$ of $X$ and maps $f,g:X\to X$, we write $(f,g)\prec\U$ and say that $f$ and $g$ are {\em $\U$-near} if for every $x\in X$ the set $\{f(x),g(x)\}$ is contained in some set $U\in\U$.

It is well-known that for any metric $d$ generating the topology of $X$, the compact-open topology on $\HH(X)$ is generated by the complete metric
$$\hat d(f,g)=\sup_{x\in X}d(f(x),g(x))+\sup_{y\in X}d(f^{-1}(y),g^{-1}(y)).$$
This implies that $\HH(X)$ is a Polish topological group, see e.g. \cite[I.\S5]{Chap}.

In the following lemmas, for subsets $F,G\subseteq X$ and a natural number $k\in\IN$, we shall establish some properties of the subset
$$H_{F,G}^k=\big\{(h_i)_{i=1}^k\in\HH(X)^k:F\subseteq\textstyle{\bigcup\limits_{i=1}^k} h_i(G)\big\}$$
of the $k$-th power $\HH(X)^k$ of the homeomorphism group $\HH(X)$.

\begin{lemma}\label{l8.4} For a closed set  $F\subseteq X$ and an open set $U\subseteq X$  the set $$H^{k}_{F,U}=\{(h_i)_{i=1}^k\in\HH(X)^k:F\subseteq\textstyle{\bigcup\limits_{i=1}^k} h_i(U)\}$$
 is open in $\HH(X)^{k}$ for every $k\in\IN$.
\end{lemma}

\begin{proof} The proof is by induction on $k\in\IN$.

First we verify this lemma for $k=1$. Since $\HH(X)$ is a topological group, it suffices to check that the set
$$
(H^1_{F,U})^{-1}=\{h\in\HH(X):F\subseteq h(U)\}^{-1}=\{h\in\HH(X):h^{-1}(F)\subseteq U\}^{-1}=\{h\in\HH(X):h(F)\subseteq U\}
$$is open in $\HH(X)$. But this follows from the definition of the compact-open topology on $\HH(X)$.

Now assume that the lemma has been proved for some $k\in\IN$. To prove it for $k+1$, fix any sequence of homeomorphisms
$(f_i)_{i=1}^{k+1}\in H^{k+1}_{F,U}$. Then $F\subseteq\bigcup_{i=1}^{k+1}f_i(U)$ and hence $\bigcap_{i=1}^{k+1}f_i(X\setminus U)\subseteq X\setminus F$.
It follows that the closed sets $B=\bigcap_{i=1}^kf_i(X\setminus U)$ and   $C=F\cap f_{k+1}(X\setminus U)$ are disjoint and hence have disjoint open neighborhoods $V$ and $O(C)$ in $X$.
Then $V$ and $W=O(C)\cup(X\setminus F)$ are two open sets such that
\begin{itemize}
\item $V\cap W\subseteq X\setminus F$;
\item $\bigcap_{i=1}^k f_i(X\setminus U)=B\subseteq V$, which is equivalent to $X\setminus V\subseteq \bigcup_{i=1}^k f_i(U)$;
\item $f_{k+1}(X\setminus U)\subseteq W$, which is equivalent $X\setminus W\subseteq f_{k+1}(U)$.
\end{itemize}
By the inductive assumption, the sets $H^k_{X\setminus V,U}\ni(f_i)_{i=1}^k$ and $H^1_{X\setminus W,U}\ni f_{k+1}$  are open and so is their product
$$H^k_{X\setminus V,U}\times H^1_{X\setminus W,U}\subseteq \HH(X)^k\times\HH(X)=\HH(X)^{k+1},$$
which contains the sequence $(f_i)_{i=1}^{k+1}$ and lies in the set $\HH^{k+1}_{F,U}$. This shows that the set $H^{k+1}_{F,U}$ is open in $\HH(X)^{k+1}$.
\end{proof}

\begin{lemma}\label{l8.5} For an $F_\sigma$-set  $F\subseteq X$ and a $G_\delta$-set $G\subseteq X$  the set $$H^{k}_{F,G}=\big\{(h_i)_{i=1}^k\in\HH(X)^k:F\subseteq\textstyle{\bigcup\limits_{i=1}^k} h_i(G)\big\}$$
 is of type $G_\delta$ in $\HH(X)^{k}$ for every $k\in\IN$.
\end{lemma}

\begin{proof} Write the $F_\sigma$-set $F$ as the union $F=\bigcup_{j\in\w}F_j$ of a non-decreasing sequence $(F_j)_{j\in\w}$ of closed subsets of $X$ and write the $G_\delta$-set $G$ as the intersection $G=\bigcap_{j\in\w}U_j$ of a non-increasing sequence $(U_j)_{j\in\w}$ of open subsets of $X$. Then for any sequence of homeomorphisms $(h_i)_{i=1}^k\in\HH(X)^{k}$ we get
$$\bigcup_{i=1}^k h_i(G)=\bigcup_{i=1}^k\bigcap_{j\in\w}h_i(U_j)=\bigcap_{j\in\w}\bigcup_{i=1}^k h_i(U_j).$$ Consequently, the set
$$
H^{k}_{F,G}=
\big\{(h_i)_{i=1}^k\in\HH(X)^{k}: \bigcup_{j\in\w}F_j\subseteq\bigcap_{j\in\w}\bigcup_{i=1}^kh_i(U_j)\big\}=\bigcap_{j\in\w}\big\{(h_i)_{i=1}^k\in\HH(X)^{k}:F_j\subseteq\bigcup_{i=0}^k h_i(U_j)\big\}=\bigcap_{j\in\w}H^k_{F_j,U_j}
$$
is of type $G_\delta$ in $\HH(X)^{k}$ as each set $H^k_{F_j,U_j}$ is open in $\HH(X)^{k}$ by Lemma~\ref{l8.4}.
\end{proof}

While the preceding two lemmas hold for any compact metrizable space $X$, the following three lemmas essentially depend on the properties of the Hilbert cube $X=\II^\w$.

\begin{lemma}\label{dense-g-delta} For any zero-dimensional $Z$-set $F\subseteq X$ and any dense open subset $G$ of $X$ the set $H_{F,G}^{-1}=\{h\in\HH(X):h(F)\subseteq G\}$ is dense in $\HH(X)$.
\end{lemma}

\begin{proof} Given a homeomorphism $h_0\in \HH(X)$ and an open cover $\U$ of the Hilbert cube $X=\II^\w$, we need to find a homeomorphism $h$ of $\II^\w$ such that $h(F)\subseteq G$ and $(h,h_0)\prec\U$.

By Lemma~\ref{l2.4}, there is a finite cover $B_1,\dots,B_n$ of the zero-dimensional $Z$-set $h_0(F)$ by open subsets of $X$ such that
\begin{itemize}
\item for every $i\le n$ the closure $\bar B_i$ of $B_i$ lies in some set $U\in\U$,
$\bar B_i$ is homeomorphic to the Hilbert cube and $\partial B_i$ is a $Z$-set in $\bar B_i$;
\item for any $1\le i\ne j\le n$ the closures $\bar B_i$ and $\bar B_j$ are disjoint.
\end{itemize}

For every $i\le n$ we shall construct a homeomorphism $f_i:\bar B_i\to\bar B_i$ such that $f_i|\partial B_i=\id$ and $f_i\big(B_i\cap h_0(F)\big)\subseteq G$. Using the Approximation Theorem~\ref{Z-approx}, in the function space $C(\II^\w,\bar B_i)$ choose a countable dense subset $\{g_j\}_{j\in\w}$ consisting of $Z$-embeddings. Then $A_i=\bigcup_{j\in\w}g_j(\II^\w)$ is a $\sigma Z$-set in $\bar B_i$ and each closed subset $C\subseteq X\setminus A_i$ of $\bar B_i$ is a $Z$-set in $\bar B_i$. Taking into account that $G$ is an open dense subset of $X$, we conclude that the set $B_i\cap G\setminus A_i$ is non-meager in $X$. Consequently, this set is uncountable, which allows us to find a topological copy of the Cantor cube in $B_i\cap U\setminus A_i$. Since the Cantor cube contains a topological copy of each zero-dimensional compact metrizable space, we can find a compact subset $K_i\subseteq B_i\cap U\setminus A_i$, homeomorphic to the compact zero-dimensional set $\bar B_i\cap h_0(F)=B_i\cap h_0(F)$. It follows from $K_i\cap A_i=\emptyset$ that $K_i$ is a $Z$-set in $\bar B_i$. Using the fact that $h_0(F)$ is a $Z$-set in $\II^\w$ and $\partial B_i$ is a $Z$-set in $\bar B_i$, we can show that $h_0(F)\cap B_i$ is a $Z$-set in $\bar B_i$. By the $Z$-set Unknotting Theorem, there is a homeomorphism $f_i:\bar B_i\to\bar B_i$ such that $f_i|\partial B_i=\id$ and $f_i(h_0(F)\cap B_i)=K_i\subseteq U$. The homeomorphisms $f_i$, $1\le i\le n$, compose a homeomorphism $f:X\to X$ such that $f|\bar B_i=f_i$, $i\le n$ and $f|X\setminus \bigcup_{i=1}^n B_i=\id$. The homeomorphism $f$ is $\U$-near to the identity homeomorphism of $X$ and $f(h_0(F))\subseteq U$. Then the homeomorphism $h=f\circ h_0:X\to X$ is $\U$-near to $h_0$ and $h(F)\subseteq U$.
\end{proof}

\begin{lemma}\label{l8.6} For every dense $G_\delta$-set $G$ in $\II^\w$ and every $\sigma Z$-set $F\subseteq \II^\w$ of finite dimension $k=\dim(F)$ the set $$H^{k+1}_{F,G}=\big\{(h_i)_{i=1}^{k+1}\in\HH(X)^{k+1}: F\subseteq\textstyle{\bigcup\limits_{i=1}^{k+1}} h_i(G)\big\}$$is dense $G_\delta$ in the space $\HH(X)^{k+1}$.
\end{lemma}

\begin{proof} By Lemma~\ref{l8.5}, the set $H^{k+1}_{F,G}$ is of type $G_\delta$ in $\HH(X)^{k+1}$. So, it remains to prove that this set is dense in $\HH(X)^{k+1}$. This will be done by induction on $k\in\IN$.

First we check the lemma for $k=0$. Fix a $\sigma Z$-subset $F$ in $X$ of dimension $\dim(F)=0$ and consider the subset
$H^1_{F,G}=\{h\in\HH(X):F\subseteq h(G)\}$ of $\HH(X)$.
Write the $\sigma Z$-set $F$ as the union $F=\bigcup_{j\in\w}F_j$ of an increasing sequence $(F_j)_{j\in\w}$ of $Z$-sets in $X$ and write the dense $G_\delta$-set $G$ as the intersection $G=\bigcap_{j\in\w}U_j$ of a decreasing sequence $(U_j)_{j\in\w}$ of open dense subsets of $X$. Observe that
$$H^1_{F,G}=\bigcap_{j\in\w}H^1_{F_j,U_j}.$$
By Lemma~\ref{dense-g-delta}, for every $j\in\w$ the set
$$(H^1_{F_j,U_j})^{-1}=\{h\in\HH(X):F_j\subseteq h(U_j)\}^{-1}=\{h\in\HH(X):h(F_j)\subseteq U_j\}$$ is dense in $\HH(X)$ and so is its inverse $H^1_{F_j,U_j}$. By Lemma~\ref{l8.4}, the set $H^1_{F_j,U_j}$ is open in $\HH(X)$.  Then the set $H^1_{F,G}=\bigcap_{i\in\w}H^1_{F_j,U_j}$ is a dense $G_\delta$ in $\HH(X)$, being a countable intersection of open dense sets in the Polish space $\HH(X)$.
\smallskip

Now assume that the lemma has been proven for some $k\in\w$. Given any $\sigma Z$-set $F\subseteq\II^\w$ of dimension $\dim(F)=k+1$, we need to prove that the set $H^{k+2}_{F,G}$ is dense in $\HH(X)^{k+2}$. Fix any non-empty open set $\U\subseteq\HH(X)^{k+2}=\HH(X)^{k+1}\times\HH(X)$. We can assume that $\U$ is of the form $\U=\V\times\W$ for some open sets $\V\subseteq\HH(X)^{k+1}$ and $\W\subseteq\HH(X)$.

The space $F$ has (inductive) dimension $k+1$ and hence $F$ has a countable base $\mathcal B=\{U_j:j\in\w\}$ of the topology such that the boundary $\partial U_j$ of each set $U_j\in\mathcal B$ in the space $F$ has dimension $\dim(\partial U_j)\le k$. By the inductive assumption, for every $j\in\w$ the set
$$H^{k+1}_{\partial U_j,G}=\big\{\{(h_i)_{i=1}^{k+1}\in\HH(X)^{k+1}:\partial U_j\subseteq\textstyle{\bigcup\limits_{i=1}^{k+1}}h_i(G)\big\}$$is dense $G_\delta$ in $\HH(X)^{k+1}$. Then the intersection $\bigcap_{j\in\w}H^{k+1}_{\partial U_j,G}$ is a dense $G_\delta$-set in $\HH(X)^{k+1}$ as well. So, we can choose a sequence of homeomorphisms $(h_i)_{i=1}^{k+1}\in \V\cap\bigcap_{j\in\w}H^{k+1}_{\partial U_j,G}$. For these homeomorphisms, we get $\bigcup_{j\in\w}\partial U_j\subseteq\bigcup_{i=1}^{k+1} h_i(G)$.

Now consider the $\sigma Z$-set $F'=F\setminus\bigcup_{i=1}^{k+1} h_i(G)$. Since $\{U_j\}_{j\in\w}$ is the base of the topology of the space $F$ and the intersection $F'\cap \bigcup_{j\in\w}\partial U_j$ is empty, the set $F'$ has dimension zero. By the inductive assumption (for $k=0$) the set $H^1_{F',G}$ is dense $G_\delta$ in $\HH(X)$, which allows us to find a  homeomorphism $h_{k+2}\in\W\cap H^1_{F',G}$. Then the sequence of homeomorphisms $(h_i)_{i=0}^{k+2}$ belongs to the set $(\V\times\W)\cap H^{k+2}_{F,G}$ witnessing that the set $H^{k+2}_{F,U}$ is dense in $\HH(X)^{k+2}$.
\end{proof}

\begin{lemma}\label{l3.5} For every dense $G_\delta$-set $G$ in $X=\II^\w$ and every $F_\sigma$-set $F\subseteq \II^\w$ of finite dimension $k=\dim(F)$ the set $$H^{k+4}_{F,G}=\big\{(h_i)_{i=1}^{k+4}\in\HH(X)^{k+4}: F\subseteq\textstyle{\bigcup\limits_{i=1}^{k+4}} h_i(G)\big\}$$is dense $G_\delta$ in the space $\HH(X)^{k+4}$.
\end{lemma}

\begin{proof} Given any non-empty open set $\U\subseteq \HH(X)^{k+4}$, we need to show that $\U\cap H^{k+4}_{F,G}\ne\emptyset$. We lose no generality assuming that $\U=\V\times\W$ for some non-empty open sets $\V\in \HH^{k+1}(X)$ and $\W\subseteq \HH^3(X)$.

By Theorem~\ref{Z-approx}, the function space $C(\II^2,X)$ contains a dense countable subset $\{f_j\}_{j\in\w}$ consisting of $Z$-embeddings. It follows that $A=\bigcup_{j\in\w}f_j(\II^2)$ is a $\sigma Z$-set of dimension $\dim(A)=2$ in $\II^\w$. By Lemma~\ref{l8.6}, the set $H^3_{A,G}$ is dense in $\HH^3(G)$. Consequently, we can find homeomorphisms $(h_{k+2},h_{k+3},h_{k+4})\in\W$ such that $A\subseteq \bigcup_{i=k+2}^{k+4}h_i(G)$. Now consider the $G_\delta$-set $G'=\bigcup_{i=k+2}^{k+4}h_i(G)$ and the finite-dimensional $F_\sigma$-set $F\setminus G'\subseteq X\setminus A$ in $X$. It follows from the choice of the set $A$ that $F\setminus G'$ is a $\sigma Z_2$-set. Since each finite-dimensional $Z_2$-set in the Hilbert cube is a $Z$-set (see \cite{Kron}), the finite-dimensional $\sigma Z_2$-set $F\setminus G'$ is a $\sigma Z$-set in $X=\II^\w$. By Lemma~\ref{l8.6}, the set $H^{k+1}_{F\setminus G',G}$ is dense in $\HH(X)^{k+1}$, which allows us to find homeomorphisms $(h_1,\dots,h_{k+1})\in\V$ such that $F\setminus G'\subseteq \bigcup_{i=1}^{k+1}h_i(G)$. Then $(h_1,\dots,h_{k+1},h_{k+2},h_{k+3},h_{k+4})\in \V\times\W=\U$ is a sequence of homeomorphisms with $F\subseteq\bigcup_{i=1}^{k+4}h_i(G)$, which means that this sequence belongs to the set $H^{k+4}_{F,G}$.
\end{proof}

\begin{proof}[Proof of Theorem~\ref{t1.4}] Let $\I\not\subseteq\M$ be any topologically invariant ideal with BP-base on the Hilbert cube $\II^\w$. Repeating the argument of the proof of Theorem~\ref{t1.1}, we can show that $\I$ contains a dense $G_\delta$-set $G$ in $\II^\w$. By Lemma~\ref{l3.5}, for any closed subset $F\subseteq \II^\w$ of finite dimension $k=\dim(F)$ there are homeomorphisms $h_1,\dots,h_{k+4}\in\HH(\II^\w)$ such that $F\subseteq \bigcup_{i=1}^{k+4}h_i(G)$. By the topological invariance and additivity of $\I$, the union $\bigcup_{i=1}^{k+4}h_i(G)$ and its subset $F$ belong to the ideal $\I$. So, $\overline{\mathcal D}_{<\w}\subseteq\I$.

If $\I$ is a $\sigma$-ideal, then $\sigma\overline{\mathcal D}_{<\w}\subseteq\I$.
\end{proof}

\section{Proof of Theorem~\ref{t1.5}}\label{s:t1.5}

The proof of Theorem~\ref{t1.5} is divided into four lemmas: \ref{cov}, \ref{non}, \ref{add}, and \ref{cof} reducing the problem of calculation of the cardinal characteristics of the ideal $\sigma\C_0$ to zero-dimensional level. The reduction will be made with help of semi-open bijection of the Baire space $\IZ^\w$ onto the Hilbert cube.

A map $f:X\to Y$ between topological spaces is called {\em semi-open} if for each non-empty open set $U\subseteq X$ the image $f(U)$ has non-empty interior in $Y$. We recall that all maps considered in this paper are continuous.
The following property of bijective semi-open maps is immediate.

\begin{lemma} If $f:X\to Y$ is a bijective semi-open map between topological spaces, then for any nowhere dense subset $A\subseteq X$ its image $f(A)$ is nowhere dense in $Y$.
\end{lemma}

A standard example of a semi-open map is the Cantor ladder map
$$c:\{0,1\}^\w\to [0,1],\;\;c:(x_i)_{i\in\w}\mapsto\sum_{i=0}^\infty \frac{x_i}{2^{i+1}}$$
of the Cantor cube $\{0,1\}^\w$ onto the closed interval $[0,1]$. This map will be used to prove:

\begin{lemma}\label{l4.2} There exists a bijective semi-open map $\varphi:\IZ^\w\to \II$ of the Baire space $\IZ^\w$ onto the closed interval $\II$.
\end{lemma}

\begin{proof} Consider the Cantor ladder map $c:\{0,1\}^\w\to[0,1]$.
It is well-known that for each point $y$ of the set $Q_2=\{\frac{m}{2^k}:0<m<2^k,\;k,m\in\w\}\subseteq [0,1]$ the preimage $c^{-1}(y)$ has cardinality $|c^{-1}(y)|=2$ and for every $y\in[0,1]\setminus Q_2$ the preimage $c^{-1}(y)$ is a singleton.
Take any subset $B\subseteq \{0,1\}^\w$ such that the restriction $c|B:B\to[0,1]$ is bijective. It follows that the set $B$ is dense and has countable dense complement in $\{0,1\}^\w$. Consequently, $B$ is zero-dimensional Polish nowhere locally compact space, which is homeomorphic to the Baire space $\IZ^\w$ according to the Aleksandrov-Urysohn Theorem~\cite[7.7]{Ke}. It is easy to see that the restriction $c|B:B\to [0,1]$ is semi-open. Then for any homeomorphism $h:\IZ^\w\to B$ the map $\varphi=c\circ h:\IZ^\w\to[0,1]$ is a bijective semi-open map of the Baire space $\IZ^\w$ onto the interval $[0,1]$.
\end{proof}

We shall consider the Baire space $\IZ^\w$ as a topological group endowed with the operation of addition of functions. In this group consider the closed nowhere dense subset $\IZ_0^\w$ where $\IZ_0=\IZ\setminus\{0\}$.

\begin{lemma}\label{l4.3} There is a bijective semi-open map $\Phi:\IZ^\w\to\II^\w$ such that for every $f\in\IZ^\w$ the set $\Phi(f+\IZ_0^\w)$ belongs to the $\sigma$-ideal $\sigma\C_0$ generated by zero-dimensional $Z$-sets in $\II^\w$.
\end{lemma}

\begin{proof} Take the bijective semi-open map $\varphi:\IZ^\w\to\II$ from Lemma~\ref{l4.2} and consider its countable power $$\varphi^\w:(\IZ^\w)^\w\to\II^\w,\;\;\varphi^\w:(x_i)_{i\in\w}\mapsto (\varphi(x_i))_{i\in\w}.$$ For each function $f\in(\IZ^\w)^\w$ the set $f+(\IZ_0^\w)^\w$ can be written as the countable product $\prod_{n\in\w}(f_n+\IZ_0^\w)$ for suitable functions $f_n\in\IZ^\w$, $n\in\w$. Then $\varphi^\w(f+(\IZ_0^\w)^\w)=\prod_{n\in\w}\varphi(f_n+\IZ_0^\w)$. Observe that for every $n\in\w$ the set $f_n+\IZ^\w_0$ is nowhere dense in $\IZ^\w$. Since the map $\varphi$ is bijective and semi-open, the image $\varphi(f_n+\IZ^\w_0)$ is nowhere dense in the interval $\II$ and so is its closure $K_n$ in $\II$. By Lemma~\ref{product-Z-set}, the product $K=\prod_{n\in\w}K_n$ is a zero-dimensional $Z$-set in $\II^\w$. Consequently, the set $\varphi^\w(f+(\IZ_0^\w)^\w)\subseteq K$ belongs to the ideal $\sigma\C_0$.
Then for any coordinate permuting homeomorphism $h:\IZ^\w\to(\IZ^\w)^\w$ the map $\Phi=\varphi^\w\circ h:\IZ^\w\to\II^\w$ has the required property: $\Phi(f+\IZ^\w_0)\in\sigma\C_0$ for every $f\in\IZ^\w$.
\end{proof}

\begin{lemma}\label{cov}
$\cov(\sigma\C_0)=\cov(\M).$
\end{lemma}

\begin{proof}
The inequality  $\cov(\sigma\C_0)\ge \cov(\M)$ is obvious, because $\sigma\C_0\subseteq \M.$
The proof of the inequality  $\cov(\sigma\C_0)\le \cov(\M)$ uses the equality
$$\cov(\M)=\min\{|\F|:\;\F\subseteq\IZ^\w\mbox{ and }\F+\IZ_0^\omega=\IZ^\w\}$$
proved in Theorem 2.4.1 \cite{BJ} (see also \cite{Miller}). According to this equality, there is a subset $\F\subseteq \IZ^{\w}$ of cardinality
$|\F|= \cov(\M)$ such that $\IZ_0^{\w}+\F=\IZ^{\w}.$

By Lemma~\ref{l4.3}, there is a bijective map $\Phi:\IZ^{\w}\rightarrow \II^\w$ such that for every $f\in\IZ^\w$ the set $\Phi(f+\IZ^\w_0)$ belongs to the ideal $\sigma\C_0$.
Since $$\II^\w=\Phi(\IZ^{\w})=\bigcup_{f\in\F}\Phi(f+\IZ^\w_0),$$
the family $\{\Phi(f+\IZ^\w_0)\}_{f\in\F}\subseteq\sigma\C_0$ is a cover of $\II^\w$, witnessing that $\cov(\sigma\C_0)\le|\F|=\cov(\M)$.
\end{proof}

\begin{lemma}\label{non}
  $\non(\sigma\C_0)=\non(\M).$
\end{lemma}
\begin{proof}

The inequality $\non(\sigma\C_0)\le \non(\M)$ is obvious, since $\sigma\C_0\subseteq \M.$
To prove the inequality  $\non(\sigma\C_0)\ge \non(\M),$ we shall use a combinatorial characterization of the cardinal $\non(\M)$ due to Bartoszynski \cite[2.4.7]{BJ}. According to this characterization, $\non(\M)$ coincides with the smallest cardinality of a subset $A\subseteq\IZ^\w$ which cannot be covered by countably many sets of the form $f+\IZ^\w_0$, $f\in\IZ^\w$.

Let $\Phi:\IZ^\w\to \II^\w$ be the bijective map from Lemma~\ref{l4.3}. Observe that for any subset $A\subseteq \II^\w$ of cardinality $|A|<\non(\M )$ its preimage $\Phi^{-1}(A)\subseteq\IZ^\w$ has cardinality $|\Phi^{-1}(A)|=|A|<\non(\M)$ and by the combinatorial characterization of $\non(\M)$, can be covered by the set $C+\IZ^\w_0$ for some countable set $C\subseteq \IZ^\w$.
Then $A\subseteq\bigcup_{f\in C}\Phi(f+\IZ^\w_0)\in \sigma\C_0$. This implies that $\non(\sigma\C_0)\ge\non(\M)$.
\end{proof}

\begin{lemma}\label{add}
$\add(\sigma\C_0,\M)=\add(\M).$
\end{lemma}

\begin{proof}
The inequality $\add(\sigma\C_0,\M)\ge\add(\M)$ is trivial.
Since $\add(\M)=\min\{\cov(\M),\mathfrak b\}$, the inequality $\add(\sigma\C_0,\M)\le \add(\M)$ will follow as soon as we check that $\add(\sigma\C_0,\M)\le\min\{\cov(\M),\mathfrak b\}$.
 Lemma~\ref{cov} implies that $\add(\sigma\C_0,\M)\le\cov(\sigma\C_0)=\cov(\M)$.

To prove that $\add(\sigma\C_0,\M)\le\mathfrak b$, consider the set $\II\setminus\IQ$ of irrational numbers in $\II$ and its countable power $(\II\setminus\IQ)^\w$, which is homeomorphic to the Baire space $\IZ^\w$ according to the Aleksandrov-Urysohn Theorem~\cite[7.7]{Ke}. By Theorem~2.2.3 of \cite{BJ}, the space $(\II\setminus\IQ)^\w$ (being a topological copy of $\IZ^\w$) contains a family $\A$ of compact subsets of cardinality $|\A|=\mathfrak b$ whose union $\bigcup\A$ is non-meager in $(\II\setminus\IQ)^\w$ and hence is non-meager in the Hilbert cube $\II^\w$. By Lemma~\ref{product-Z-set}, each set $A\in\A$ is a zero-dimensional $Z$-set in $\II^\w$ and hence $\A\subseteq\sigma\C_0$. Since $\bigcup\A\notin\M$, we see that $\add(\sigma\C_0,\M)\le|\A|=\mathfrak b$.
\end{proof}

\begin{lemma}\label{cof}
 $\cof(\sigma\C_0,\M)=\cof(\M).$
\end{lemma}

\begin{proof}
The inequality $\cof(\sigma\C_0,\M)\le \cof(\M)$ is trivial.  Since $\cof(\M)=\max\{\non(\M),\mathfrak d\}$, the inequality $\cof(\sigma\C_0,\M)\ge \cof(\M)$ will follow as soon as we check that
$\cof(\sigma\C_0,\M)\ge\max\{\non(\M),\mathfrak d\}$.

To see that $\cof(\sigma\C_0,\M)\ge\non(\M)$ we shall use the equality $\non(\M)=\non(\sigma\C_0)$ proved in Lemma~\ref{non}. By the definition of the cardinal $\cof(\sigma\C_0,\M)$, there is a subfamily $\A\subset\M$ of cardinality $|\A|=\cof(\sigma\C_0,\M)$ such that every subset $C\in\sigma\C_0$ is contained in some set $A\in\A$. For every $A\in\A$ choose a point $x_A\in\II^\w\setminus A$ and observe that the set $X=\{x_A\}_{A\in\A}$ is contained in no set $A\in\A$, which implies that $X\notin\sigma\C_0$ and hence $\non(\M)=\non(\sigma\C_0)\le|X|\le |\A|=\cof(\sigma\C_0,\M)$.

To prove that $\cof(\sigma\C_0,\M)\ge \mathfrak d$, consider the set $\II\setminus\IQ$ of irrational numbers in $\II$ and its countable power $P=(\II\setminus\IQ)^\w$, which is homeomorphic to the Baire space $\IZ^\w$ according to the Aleksandrov-Urysohn Theorem~\cite[7.7]{Ke}. Let $\M(P)$ be the ideal of meager sets in $P$ and $\sigma\K(P)$ be the $\sigma$-ideal generated by compact subsets of $P$. By Theorem~2.2.3 of \cite{BJ}, $\cof(\sigma\K(P),\M(P))=\mathfrak d$.
Taking into account that $\sigma\K(P)\subseteq\sigma\C_0$ (which follows from Lemma~\ref{product-Z-set}) and  $\M(P)=\{M\cap P:M\in\M\}$, we see that
$$\mathfrak d=\cof(\sigma\K(P),\M(P))\le\cof(\sigma\C_0,\M).$$
Therefore, $\cof(\sigma\C_0,\M)\ge\max\{\non(\M),\mathfrak d\}=\cof(\M)$ and we are done.
\end{proof}

\section{Proof of Theorem~\ref{t1.6}}\label{s:t1.6}

First we elaborate some tools for working with tame $G_\delta$-sets in the Hilbert cube $\II^\w$. We shall need an index-free description of tame $G_\delta$-sets, developed in \cite{BR2}.

A family $\Tau$ of open subsets of a topological space $X$ is defined to be {\em tame} if
\begin{itemize}
\item $\Tau$ is {\em vanishing} in the sense that for each open cover $\U$ of $X$ the family $\{B\in\Tau:\forall U\in\U\;\;B\not\subseteq U\}$ is locally finite;
\item for any distinct sets $A,B\in\Tau$ one of three possibilities holds: $\bar A\cap\bar B=\emptyset$, $\bar A\subseteq B$, or $\bar B\subseteq A$.
\end{itemize}

For a family $\Tau$ of subsets of a set $X$ consider the set
$$\mbox{$\bigcup^\infty\Tau=\bigcap\big\{\bigcup(\Tau\setminus\F):\F$ is a finite subfamily of $\Tau\big\}$}
$$
of all points $x\in X$ that belong to infinitely many sets of the family $\Tau$.

The following characterization of tame $G_\delta$-sets was proved in Proposition 2 of \cite{BR2}.

\begin{proposition}\label{p5.1} A subset $G\subseteq\II^\w$ is a tame $G_\delta$-set in $\II^\w$  if and only if $G=\bigcup^\infty\Tau$ for a tame family $\Tau$ of tame open balls in $\II^\w$.
\end{proposition}

A $G_\delta$-subset $G\subseteq\II$ will be called a {\em tame $G_\delta$-set} in $\II$ if
for any non-empty open set $U\subseteq\II$ the complement $U\setminus G$ is uncountable.

To establish some structural properties of tame $G_\delta$-sets in $\II$, we need indexed modifications of the notions of vanishing and disjoint families. An indexed family $(X_\alpha)_{\alpha\in A}$ of subsets of a compact metrizable space $X$ is called
\begin{itemize}
\item {\em disjoint} if $X_\alpha\cap X_\beta=\emptyset$ for any distinct indexes $\alpha,\beta\in A$;
\item {\em vanishing} if for each open cover $\U$ of $X$ there set $\{\alpha\in A:\forall U\in\U\;\;X_\alpha\not\subseteq U\}$ is finite.
\end{itemize}

\begin{lemma}\label{l5.2} If $G\subseteq\II$ is a dense tame $G_\delta$-set in $\II$, then for any non-empty open connected subset $U\subsetneqq \II$ and any $\e>0$ there is a sequence $(U_m)_{m\in\w}$ of non-empty open connected subsets of\/ $\II$ such that
\begin{enumerate}
\item the indexed family $(\bar U_m)_{m\in\w}$ is disjoint;
\item $\bigcup_{m\in\w}\bar U_m\subseteq U$;
\item $\diam(U_m)<\e$ for all $m\in\w$;
\item $U\cap G\subseteq\bigcup_{m\in\w}U_m$.
\end{enumerate}
\end{lemma}

\begin{proof} Being a proper open connected subset of $\II$, the set $U$ is equal to $(a,b)$, $[0,b)$, or $(b,1]$ for some numbers $0\le a<b\le 1$. So, we can choose a disjoint sequence $(V_m)_{m\in\w}$ of non-empty open connected subsets $\II$  such that
\begin{itemize}
\item[(a)]  $\bigcup_{m\in\w}V_m$ is dense in $U$;
\item[(b)] $\bigcup_{m\in\w}\bar V_m\subseteq U$;
\item[(c)] the family $\{V_m\}_{m\in\w}$ is locally finite in $U$;
\item[(d)] $\diam V_m<\e/2$ for all $m\in\w$.
\end{itemize}

Since the $G_\delta$-set $G$ is tame, for every $m\in\w$ the complement $V_m\setminus G$ is uncountable, and hence contains a topological copy $K_m$ of the Cantor cube $K_m$. The condition (c) implies that the union $K=\bigcup_{m\in\w}K_m$ is a closed subset without isolated points in $U$. Consequently, its complement $U\setminus K$ can be written as the countable union $\bigcup_{m\in\w}U_m$ of a disjoint sequence $(U_m)_{m\in\w}$ of open connected subsets of $\II$ such that the family $(\bar U_m)_{m\in\w}$ is disjoint. The condition (d) guarantees that $\diam(U_m)<\e$ for all $m\in\w$. The obvious inclusion $U\cap G\subseteq U\setminus K=\bigcup_{m\in\w}U_m$ completes the proof of the lemma.
\end{proof}

Using Lemma~\ref{l5.2}, by a standard inductive argument, one can prove:

\begin{lemma}\label{l5.3} If a dense $G_\delta$-subset $G$ of $\II$ is tame, then
$G=\bigcap_{n\in\w}\bigcup_{m\in\w}U_{n,m}$ for some vanishing indexed family $(U_{n,m})_{n,m\in\w}$ of open connected subsets of $\II$ such that for every $n\in\IN$ the indexed family $(\bar U_{n,m})_{m\in\w}$ is disjoint and the family $\{\bar U_{n+1,m}\}_{m\in\w}$ refines the family $\{U_{n,m}\}_{m\in\w}$.
\end{lemma}

\begin{lemma}\label{l5.4} For each uncountable cardinal $\kappa\le\mathfrak c$ there is a family $(G_\alpha)_{\alpha\in\kappa}$ of dense tame $G_\delta$-sets in $\II$ such that each subset $X\subseteq \II$ of cardinality $|X|<\kappa$ is contained in some set $G_\alpha$, $\alpha\in\kappa$.
\end{lemma}

\begin{proof} Fix a countable base $(U_n)_{n\in\w}$ of the topology of the interval $\II$ and in each set $U_n$ fix a disjoint family $(C_{n,\alpha})_{\alpha\in\kappa}$ of $\kappa$ many Cantor sets. Observe that for every $\alpha\in\kappa$ the complement $G_\alpha=\II\setminus\bigcup_{n\in\w}C_{n,\alpha}$ is a dense tame $G_\delta$-set in $\II$.

Given any subset $X\subseteq \II$ of cardinality $|X|<\kappa$, for every $n\in\w$ consider the set $A_n=\{\alpha\in\kappa: X\cap C_{n,\alpha}\ne\emptyset\}$ and observe that it has cardinality $|A_n|\le|X|<\kappa$. Then the union $A=\bigcup_{n\in\w}A_n$ also has cardinality $|A|<\kappa$ and we can choose an ordinal $\alpha\in\kappa\setminus A$. For this ordinal $\alpha$ we get $X\subseteq \II\setminus\bigcup_{n\in\w}C_{n,\alpha}=G_\alpha$.
\end{proof}

Now we are ready to prove the principal ingredient of the proof of Theorem~\ref{t1.6}.

\begin{proposition}\label{p5.5} Let $G$ be a dense tame $G_\delta$-set in the unit interval $\II$. Then:
\begin{enumerate}
\item the countable power $G^\w$ can be covered by $\mathfrak b$ many tame $G_\delta$-sets in $\II^\w$;
\item any subset $X\subseteq G^\w$ of cardinality $|X|<\mathfrak d$ can be covered by a single tame $G_\delta$-set in $\II^\w$.
\end{enumerate}
\end{proposition}

\begin{proof} By Lemma~\ref{l5.3}, $G=\bigcap_{n\in\w}\bigcup_{m\in\w}U_{n,m}$ for some vanishing indexed family $\U=(U_{n,m})_{n,m\in\w}$ of open connected subsets of $\II$ such that for every $n\in\IN$ the indexed family $(\bar U_{n,m})_{m\in\w}$ is disjoint and the family  $\{\bar U_{n+1,m}\}_{m\in\w}$ refines the family $\{U_{n,m}\}_{m\in\w}$.

For every increasing function $f:\w\to\w$ we define a tame $G_\delta$-set $\Tau^f$ in $\II^\w$ as follows. For every $n\in\w$ let
$$
\begin{aligned}
a^f_n&=\max\{\sup U_{n,m}:m\le f(n),\;1\notin \bar U_{n,m}\}\mbox{ and }\\
b^f_n&=\min(\{1\}\cup \{\inf U_{n,m}:m\le f(n),\;1\in\bar U_{n,m}\}).
\end{aligned}
$$
Since the indexed family $(\bar U_{n,m})_{m\in\w}$ is disjoint, $0<a^f_n<b^f_n\le 1$. Moreover, $\bigcup_{m\le f(n)}U_{n,m}\subseteq [0,a_n^f)\cup(b_n^f,1]$.

Now for every $n\in\w$ consider the finite family
$$\Tau^f_n=\Big\{\prod_{i\in\w}V_i:V_0,\dots,V_{n-1}\in\{U_{n,m}\}_{m\le f(n)},\;V_n\in\{[0,a^f_n),(b^f_n,1]\},\;V_i=\II\mbox{ for all }i>n\Big\}$$
of tame open balls in the Hilbert cube $\II^\w$.

\begin{claim}\label{cl5.6} The family $\{\bar V:V\in\Tau^f_n\}$ is disjoint.
\end{claim}

\begin{proof} This follows immediately from the fact that the family $\{\bar U_{n,m}\}_{m\in\w}$ is disjoint.
\end{proof}

\begin{claim} The family $\Tau^f=\bigcup_{n\in\w}\Tau_n^f$ is tame.
\end{claim}

\begin{proof} We need to check two conditions from the definition of a tame family.
\smallskip

1. The vanishing property of  $\Tau$ will follow as soon as we check that for each $\e>0$ the subfamily $\{V\in\Tau:\diam(V)\ge\e\}$ is finite. Here we consider the metric
$$d\big((x_n)_{n\in\w},(y_n)_{n\in\w}\big)=\max_{n\in\w}\frac{|x_n-y_n|}{2^n}$$
on the Hilbert cube $\II^\w$.

Since the double sequence $(U_{n,m})_{m\in\w}$ is vanishing, there is $n\in\w$ so large that $2^{-n}<\e$ and $\diam(U_{k,m})<\e$ for all $k\ge n$ and all $m\in\w$. Then for every $k\ge n$, every tame open ball $V\in\Tau_k^f$ has diameter $\diam(V)<\e$. This implies that the family $\{V\in\Tau^f:\diam(V)\ge \e\}\subseteq\bigcup_{k<n}\Tau^f_k$ is finite.
\smallskip

2. Given two distinct sets $V,W\in \Tau^f$, we need to check that $\overline{V}\cap\overline{W}=\emptyset$, $\overline{V}\subseteq W$ or $\overline{W}\subseteq V$. Find numbers $k,n\in\w$ such that $V\in\Tau_k^f$ and $W\in\Tau^f_n$. We lose no generality assuming that $k\le n$. If $n=k$, then $\bar V\cap\bar W=\emptyset$ by Claim~\ref{cl5.6}. Next, consider the case $k<n$. It follows that $V=\prod_{i\in\w}V_i$ and $W=\prod_{i\in\w}W_i$ for some sets $V_0,\dots,V_{k-1}\in \{U_{k,m}\}_{m\le f(k)}$, $V_k\in\{[0,a^f_k),(b_k^f,1]\}$, $V_i=\II$ for $i>k$, and
$W_0,\dots,W_{n-1}\in \{U_{n,m}\}_{m\le f(n)}$, $W_n\in\{[0,a^f_n),(b_n^f,1]\}$, and $W_i=\II$ for $i>n$.
The choice of the family $\{U_{i,j}\}_{i,j\in\w}$ guarantees that either $\bar V_i\cap \bar W_i=\emptyset$ for some $i<k$ or else $\bar W_i\subseteq V_i$ for all $i<k$. In the first case the sets $V$ and $W$ have disjoint closures. So, it remains to consider the second case:  $\bar W_i\subseteq V_i$ for all $i<k$. Consider the set $W_k$, which is equal to $U_{n,m}$ for some $m\le f(n)$.
It follows that $\bar U_{n,m}\subseteq U_{k,m'}$ for some number $m'\in\w$.
Since the family $\{\bar U_{k,i}\}_{i\in\w}$ is disjoint and consists of connected subsets of $\II$ four cases are possible: (i) $U_{k,m'}\subseteq [0,a^f_k)$,  (ii) $b^f_k<1$ and $U_{k,m'}\subseteq (b^f_k,1]$, (iii) $b^f_k<1$ and $\bar U_{k,m'}\subseteq (a^f_k,b^f_k)$, and (iv) $b^f_k=1$ and $\bar U_{k,m'}\subseteq (a^f_k,1]$. In the case (i) we get $\overline{W}\subseteq V$ if $V_k=[0,a^f_k)$ and $\overline{W}\cap\overline{V}=\emptyset$ if $V_k=(b^f_k,1]$.  In the case (ii) we get $\overline{W}\subseteq V$ if $V_k=(b^f_k,1]$ and $\overline{W}\cap\overline{V}=\emptyset$ if  $V_k=[0,a^f_k)$. In the cases (iii) and (iv) we get $\overline{W}\cap\overline{V}=\emptyset$.
\end{proof}

Since $\Tau^f$ is a tame family consisting of tame open balls in the Hilbert cube $\II^\w$, the set $T^f=\bigcup^\infty\Tau^f$ is a tame $G_\delta$-set in $\II^\w$ by Proposition~\ref{p5.1}.

For each point $x=(x_i)_{i\in\w}\in G^\w$, consider the function $f_x:\w\to\w$ assigning to each number $n\in\w$ the smallest number $f_x(n)$ such that $x_0,\dots,x_n\in \bigcup_{m\le f_x(n)}U_{n,m}$.

\begin{claim}\label{cl5.8} For a function $f\in\w^\w$ and a point $x\in G^\w$ with $f\not\le^*f_x$ we get $x\in T^f$.
\end{claim}

\begin{proof} It follows that for every $n\in\w$ with $f_x(n)< f(n)$, we get
$$x_0,\dots,x_n\in \bigcup_{m\le f_x(n)}U_{n,m}\subseteq \bigcup_{m\le f(n)}U_{n,m},$$ which implies that $x=(x_i)_{i\in\w}\in \bigcup\Tau^f_n$ and hence $x\in\bigcup^\infty \Tau^f=T^f$ as the set $\{n\in\w:f_x(n)<f(n)\}$ is infinite.
\end{proof}

Now we can complete the proof of Proposition~\ref{p5.5}.

1. By the definition of the cardinal $\mathfrak b$, there is a subset $\F\subseteq\w^\w$ of cardinality $|\F|=\mathfrak b$ such that for every $g\in\w^\w$ there is $f\in\F$ with $f\not\le^* g$.  Then for any point $x\in G^\w$ there is a function $f\in\F$ such that $f\not\le^* f_x$. By Claim~\ref{cl5.8}, $x\in T^f\subseteq \bigcup_{g\in\F}T^g$, which means that $G^\w$ is covered by $\mathfrak b$ many tame $G_\delta$-sets $T^g$, $g\in\F$, in $\II^\w$.
\smallskip

2. If $X\subseteq G^\w$ is a subset of cardinality $|X|<\mathfrak d$, then the set $\{f_x:x\in X\}$ is not dominating in $\w^\w$ and hence there is a function $f\in\w^\w$ such that $f\not\le^* f_x$ for all $x\in X$. By Claim~\ref{cl5.8}, $x\in T^f$ for each $x\in X$, which means that $X$ is covered by the tame $G_\delta$-set $T^f$.
\end{proof}

The following two lemmas imply Theorem~\ref{t1.6}.

\begin{lemma} $\cov(\sigma\G_0)\le\add(\M)$.
\end{lemma}

\begin{proof} Since $\add(\M)=\min\{\cov(\M),\mathfrak b\}$, it suffices to prove that  $\cov(\sigma\G_0)\le\min\{\cov(\M),\mathfrak b\}$. The inequality $\cov(\sigma\G_0)\le\cov(\M)$ trivially follows from the inclusion $\sigma\C_0\subseteq\sigma\G_0$ and the equality $\cov(\sigma\C_0)=\cov(\M)$ proved in Lemma~\ref{cov}.

To prove that $\cov(\sigma\G_0)\le\mathfrak b$, apply Lemma~\ref{l5.4} and find an uncountable family $(G_\alpha)_{\alpha\in\w_1}$ of dense tame $G_\delta$-sets in $\II$ such that each countable subset of $\II$ is contained in some $G_\alpha$. This implies that $\II^\w=\bigcup_{\alpha\in\w_1}G_\alpha^\w$.
By Proposition~\ref{p5.5}(1), each set $G_\alpha^\w$ can be covered by $\mathfrak b$ tame $G_\delta$-subsets of the Hilbert cube. Consequently, $\II^\w$ can be covered by $\w_1\times\mathfrak b=\mathfrak b$ tame $G_\delta$-subsets of the Hilbert cube, which means that $\cov(\sigma\G_0)\le \mathfrak b$.
\end{proof}

\begin{lemma} $\non(\sigma\G_0)\ge\cof(\M)$.
\end{lemma}

\begin{proof} Since $\cof(\M)=\max\{\non(\M),\mathfrak d\}$, it suffices to prove that  $\non(\sigma\G_0)\ge\max\{\non(\M),\mathfrak d\}$. Then inequality $\non(\sigma\G_0)\ge\non(\M)$ trivially follow from the inclusion $\sigma\C_0\subseteq\sigma\G_0$ and the equality $\non(\sigma\C_0)=\non(\M)$ proved in Lemma~\ref{non}.

To prove that $\non(\sigma\G_0)\ge\mathfrak d$, fix any subset $X\subseteq \II^\w$ of cardinality $|A|<\mathfrak d$. Then $X\subseteq Y^\w$ for some subset $Y\subseteq \II$ of cardinality $|Y|\le\aleph_0\cdot|X|<\mathfrak d\le\mathfrak c$. By Lemma~\ref{l5.4}, $Y$ is contained in some dense tame $G_\delta$-set $G\subseteq\II$. By Proposition~\ref{p5.5}(2), the set $X\subseteq G^\w$ can be covered by a tame $G_\delta$-subset of the Hilbert cube $\II^\w$, which implies that $\non(\sigma\G_0)\ge\non(\G_0)\ge\mathfrak d$.
\end{proof}

\section{Open Problems}

In this section we collect some open problems on topologically invariant $\sigma$-ideals on $\II^\w$.
The most intriguing problems concern the $\sigma$-ideal $\sigma\G_0$.

\begin{problem} Is $\add(\sigma\GG_0)=\cov(\sigma\GG_0)=\w_1$ and $\non(\sigma\GG_0)=\cof(\sigma\GG_0)=\mathfrak c$?
\end{problem}

\begin{problem} Is $\cov(\sigma\GG_0)=\mathfrak c$ under Martin's Axiom? Under PFA?
\end{problem}

\begin{problem} Is $\sigma\GG_0=\sigma\mathcal D_0$?
\end{problem}

It Proposition 4 of \cite{BR2} it was proved that for any dense $G_\delta$-set $G\subseteq\II$ the countable power $G^\w$ does not belong to the family $\GG_0$ of minimal dense $G_\delta$-sets in $\II^\w$.

\begin{problem} Is  $G^\w\in\sigma\G_0$ for some dense $G_\delta$-set $G\subseteq \II$?
\end{problem}

\begin{problem} Let $\I$ be a maximal non-trivial topologically invariant $\sigma$-ideal with Borel base on $\II^\w$. Is $\I=\M$?
\end{problem}

A closed subset $A\subseteq\II^\w$ is called a {\em homological $Z_\w$-set} in $\II^\w$ if $A\times\{0\}$ is a $Z_\w$-set in $\II^\w\times[-1,1]$. Let $\mathcal B$ be the family of all Borel subsets $B\subseteq\II^\w$ such that the closure $\bar C$ of each connected subset $C\subseteq B$ in $\II^\w$ is a homological $Z_\w$-set in $\II^\w$. It follows from Main Lemma of \cite{BC} that the $\sigma$-ideal $\sigma\mathcal B$ generated by the family $\mathcal B$ is non-trivial.

\begin{problem} Is the ideal $\sigma\mathcal B$ a maximal non-trivial topologically invariant $\sigma$-ideal with Borel base on $\II^\w$?
\end{problem}

\section{Acknowledgment}

The authors would like to thank S\l awomir Solecki for turning their attention to the $G_\delta$-ideals with the property $(*)$ and the Tukey reducibility of such ideals to the ideal $\NWD$.

\end{document}